\documentclass[11 pt]{amsart}
\usepackage{amssymb}
\usepackage{graphicx}
\usepackage{color}
\usepackage{amsmath}
\usepackage{graphicx}
\usepackage{mathrsfs}
\usepackage{amsfonts} 
\usepackage{amscd} 
\usepackage{epic}

\newtheorem{theorem}{Theorem}[section]

\newtheorem{definition}[theorem]{Definition}

\newtheorem{lemma}[theorem]{Lemma}
\newtheorem{remark}[theorem]{Remark}
\newtheorem{proposition}[theorem]{Proposition}

\newtheorem{example}[theorem]{Example}
\numberwithin{equation}{section}

\renewcommand{\H}{{\mathcal H}}
\def\C{\mathbb C}
\def\R{\mathbb R}

\def\C{\mathbb C}
\def\R{\mathbb R}

\def\N{\mathbb N}
\def\al{\alpha}

\def\be{\beta}

\def\rh{\rho}

\def\th{\theta}

\def\GA{\Gamma}

\def\ve{\varepsilon}

\def\la{\lambda}

\def\om{\omega}
\def\va{\varphi}
\def\ta{\tau}

\def\va{\varphi}

\def\g{\mathfrak{g}}

\def\g{\mathfrak g}

\def\la{\lambda}
\def\ve{\varepsilon}
\def\si{\sigma}
\def\om{\omega}

\def\ga{\gamma}
\def\ph{\phi}
\def\ch{\chi}
\def\ta{\tau}

\def\N{\mathbb{N}}
\def\Z{\mathbb{Z}}
\def\R{\mathbb{R}}
\def\C{\mathbb{C}}

\def\ol#1{\overline{#1}}
\def\nn{\nonumber}
\def\noop#1{\Vert #1\Vert_{\rm op}}

\def\R{{\mathbb R}}
\def\C{{\mathbb C}}
\def\N{{\mathbb N}}

\def\Z{{\mathbb Z}}
\def\T{{\mathbb T}}

\def\A{{\mathcal A}}
\def\B{{\mathcal B}}
\def\D{{\mathcal D}}
\def\F{{\mathcal F}}

\def\H{{\mathcal H}}

\def\K{{\mathcal K}}

\def\iy{\infty}

\def\ol#1{\overline{#1}}

\def\hb#1{\hbox{#1}}

\def\no#1{\Vert #1\Vert }
\def\ey{\emptyset}

\def\wh#1{\widehat{#1}}

\def\exp#1{{\rm exp}( #1)}

\def\ker#1{\hb{ker}(#1)}

\def\res#1{_{\vert #1}}
\def\inv{^{-1}}

\def\limk{\lim_{k\to\infty}}
\def\limm{\lim_{m\to\infty}}

\def\limn{\lim_{n\to\infty}}

\def\hb#1{\hbox{#1}}

\def\ker#1{\hb{ker}(#1)}

\def\dim#1{\rm{dim}(#1)}

\def\L1#1{L^1(#1)}

\def\L#1#2{L^{#1}(#2)}
\def\l#1#2{L^{#1}(#2)}

\def\ti{\times }

\def\lef({\left(}
\def\rig){\right)}

\begin{document}
\title{ The  $C^*$-algebra of the semi-direct product $K\ltimes A $.}
\author{Regeiba Hedi And Ludwig Jean }
\address{ Universit\'{e} de Sfax, Facult\'{e} des Sciences Sfax,  BP
1171, 3038 Sfax, Tunisia\\
Université de Gabés
Facult\'e des Sciences de Gab\'es
Cit\'e Erriadh 6072 Zrig Gab\'es Tunisie.
} \email{rejaibahedi@gmail.com}
\address{ Universit\'e de Lorraine,
Institut Elie Cartan de Lorraine,
UMR 7502, Metz, F-57045, France.}
\email{jean.ludwig@univ-lorraine.fr}

\begin{abstract}
Let $G=K\ltimes A$ be the  semi-direct product group of a compact group  $K$ acting on  an abelian locally compact group  $A$.
We describe the $C^*$-algebra
$C^*(G)$ of $G$ in terms of an algebra of operator fields defined over the spectrum of $G $, generalizing previous results obtained for some special classes of such groups.

\end{abstract}
 \maketitle
\section{\bf{Introduction}}\label{sec:1}

It is wellknown that for a simply connected nilpotent Lie group and
 more
generally for an exponential solvable Lie group $ G=\exp\g $, its  dual
 space $
\widehat G$ is homeomorphic to the space of co-adjoint orbits $ \g^*/G $
 through
the Kirillov mapping (see \cite{Lep-Lud}). If we consider semi-direct
 products
$ G=K\ltimes N $ of compact connected Lie groups $ K $ acting on simply
connected nilpotent Lie groups $ N $, then again we have an orbit
 picture of the
dual space of $ G $ (see \cite{Lipsman}) and one can imagine that the
topology of $ \widehat G $ is linked to the topology of the admissible
 co-adjoint
orbits.\\
By definition the $C^*$-algebra $C^*(G)$  of a locally compact group $G$ is the
completion of the convolution algebra $L^1(G)$ with respect to the norm
$$\no{f}_{C^*(G)}:=\underset{\pi\in\widehat G}{\sup}\noop{\pi(f)}.$$
The unitary dual or spectrum $\widehat{C^*(G)}$ of $C^*(G)$ is in bijection with the dual space $\widehat G$ of $G$. Define the Fourier transform $\F$  on $C^*(G)$ by
$$\F(c)(\pi)=\pi(c)\in\B(\H_\pi),\ \text{ for all } \pi\in\widehat G,\ c\in C^*(G).$$
Then $C^*(G)$
can be identified with the  sub-algebra $\widehat{C^*(G)} $ of the big $C^*$-algebra
$\ell^\iy(\widehat G)$ of bounded operator fields given by
$$\ell^\iy(\widehat G)=
\left\{\ph:\widehat G\longrightarrow\underset{\pi\in\widehat G}{\bigcup}\B(\H_\pi);\ph(\pi)\in \B(\pi), \no \ph_\iy:=\underset{\pi\in\widehat G}{\sup}\noop{f(\pi)}<\iy\right\}.$$
Here $\B(\H) $ denotes the space of bounded linear operators on the Hilbert space $\H $.

In order to understand the structure of the algebra $C^*(G) $, we must determine the special conditions which determine the operator fields $\ph\in \widehat{C^*(G)}
$.

For instance if $G=K $, then it is easy to see that
\begin{eqnarray*}\label{}
 \nn \widehat{C^*(K)}
 &= &
 \left\{F:\widehat K\longrightarrow\underset{\pi\in\widehat K}{\bigcup}\B(\H_\pi)|\ F(\pi)\in\B(\H_\pi),\   \lim_{\pi\to\infty}\noop{ F(\pi)}=0\right\}.
 \end{eqnarray*}
Since $K $ is compact, its irreducible unitary representations are finite dimensional. So $\pi(c) $ is always trivially a compact operator. Furthermore, the spectrum of $K $ has the discrete topology. So there is no continuity condition for the operator fields.

If $G=A $ is abelian, then $\widehat A $ is  a   locally compact Hausdorff space and
\begin{eqnarray*}\label{}
 \nn \widehat{C^*(A)}
&= &
 \left\{ \ph: \widehat A\to \C; \ph \text{ continuous and vanishes at infinity}\right\}.
 \end{eqnarray*}
 If $G$ is more generally a locally compact  group of the form $K\ltimes A $, then its dual space $\widehat{G} $
has a complicated structure determined by the $K $-orbits in $\widehat A $ and the spectra of the stabilizer groups of the elements of $\widehat A $. The topology of $\widehat{G} $ has been described in the paper of W. Bagget  \cite{Bag}, which will be used  intensively here.
For the motion groups of the form  $G_n=SO(n)\ltimes\R^n$
the conditions for $\widehat{C^*(G_n)} $ have been described explicitly in \cite{Lud-Ell-Abd}.
For some other groups only a detailed description of the topology of $\widehat G $ is known
(see for instance \cite{Ben-Rah} and \cite{Re-Ra}).

 In section $2$, we recall the  results of Bagget on the topology of $\widehat{K\ltimes A} $, define the Fourier transform for $C^*(G)  $  and discover the conditions which determine the algebra $\widehat{C^*(G)}
 $
inside the big algebra $\ell^\iy(\widehat G)$ (see Definition \ref{def0.3.2}).
To see what are the difficulties, consider a converging net $\Pi_M=(\pi_m=\pi_{\mu_m,\ch_m})_{m\in M} $ of irreducible representations of $G $. Here $(\ch_m)_{m\in M} $ is a net of unitary characters of the abelian group  $A $, $K_{m}$ is  the stabilizer of $\ch_m $ in $K $ and  $\mu_m $ is an element of $\widehat{K_m}, m\in M$.
Then  $\pi_m=\text{ind}_{K_m\ltimes A}^G \mu_m\otimes \ch_m $.
The problem is to  understand for $a\in C^*(G) $ the behavior  of the net of operators $(\pi_m(a))_{m\in M} $ acting on the  different Hilbert spaces $\H_{\pi_m} $.
We shall  in fact construct for a certain converging subnet of the net $\Pi_M $ a common Hilbert space $\H_M $, such that $\H_M  $ contains a copy $\H_m' $ of the Hilbert space $\H_m $ of $\pi_m $ and  a projection $P_m:\H_M\to \H_m' $ for every $m$ in the subnet and we show that  the essential  condition for an operator field $\Phi $ to be an element of  $\widehat{C^*(G)} $ is an operator norm convergence of the net  $(\Phi(\pi_m))_{m\in M} $ to a limit operator $\si(\Phi\res L)\in \B(\H_M) $ which is determined by the restriction $\Phi\res L $ of the operator field $\Phi $ to the limit set $L $ of the subnet.

In the last paragraph we present as example  the group
$G_{n,m}=SO(n)\times SO(m)\ltimes \R^{n+m}$, for $n,m\in\N$.

  \section{\bf{The $C^*$-algebra of the group $G=K\ltimes A.$}}
  \subsection{\bf{Preliminaries.}}
  Let $K$ be locally compact group, $A$ be an abelian   and suppose $\Psi$ is homomorphism of $K$ into the group
  of automorphisms of $A$ such that the mapping $\left.\begin{array}{cccc}
 & & & \\
\Psi: & K\ti A& \longrightarrow& A\\
&(k,a)&\mapsto&\Psi(k)(a)\end{array}\right.$ is continuous. For simplicity, we write the action of the automorphism $\Psi(k)$ on an
element $a$ of $A$ as $k\cdot a.$

The semi-direct product $G=K\ltimes A$ of the groups $K$ and $A$ is  the following locally compact group. $G$ is the topological product of $K $ and $A $ and $G $ is equipped with the group law
\begin{eqnarray*}
 (k,a)\ldotp(h,b)=(kh,h\inv\cdot a+ b),\ \forall (k,a),\ (h,b)\in G,
 \end{eqnarray*}
 where we write the multiplication  in $K$ with the symbol $\cdot  $  and multiplication  in $A$ additively as $+ $.

 The group $A$ can be homeomorphically and isomorphically identified with the closed normal subgroup of $G$ consisting
 of all pairs $(1_K,a)$, where $a$ is an element of $A$. Also, $K$ is homeomorphic and isomorphic to the closed subgroup of $G$
 consisting of all pairs $(k,0_A)$, where $k$ is an element of $K$. (We write $1_K$ for the multiplicative identity of $K$, $0_A$ for the additive
 identity of $A$, and $(1_K,0_A)$ for the identity of $G$).

 Now, let $k\in K$ and $\chi: A\to \T\in \widehat A$  be  a unitary character of  $A$, and let $k\cdot\chi$ to be the character of  $A$ defined by
 $$k\cdot\chi(b)=\chi{(k\inv\cdot  b}), b\in A.$$
 Thus $K$ acts on the  left as a group of continuous transformations on $\widehat A$. If $\chi$ is in $\widehat A$,
 the stability subgroup $K_\chi$ of $\chi$ is the closed subgroup
 $$K_\chi=\{k\in K|\ k\cdot \ch=\ch\}.$$
 \begin{definition}\label{def0.1}$ $
  \begin{enumerate}
   \item We define the space  $\K(G)$ to be the collection of all closed subgroups of $G$ equipped with the compact-open topology $($see \cite{Fell-2}$)$.
   \item Let $\A(G)$ denote the set of all pairs $(C,T)$ where $C$ is a closed subgroup of $G$ and $T$ is an irreducible unitary representation of $C.$
  \end{enumerate}

 \end{definition}

\begin{proposition}\label{prop0.1}[see, \cite{Bag} 2.1-D]
 Any closed subgroup of $G$ which contains $A$ is of the form $J\ltimes A$, where $J$ is a closed subgroup of $K$.
 Further, a net $(J_n)_n$ of closed subgroups of $K$ converges to $J$ in $\K(K)$ if and only if the net
 $(J_n\ltimes A)_n$ of subgroups of $G$ converges to the subgroup $J\ltimes A$ in $\K(G).$
\end{proposition}
Let now $\chi\in \widehat A$. Then the Hilbert space of the representation $\delta=\text{ind}_{\{1_K\}\ltimes A}^G\chi$  can be identified with $L^2(K)$,
thus if $(k,a)\in G$, $f\in L^2(K)$ and $h\in K$
we have
\begin{eqnarray}\label{calcdelta}
 \nn\delta(k,a)(f)(h)&=&\delta((k,0_A)(1_K, a))(f)(h)\\
 \nn&=&\chi((h\inv k)\cdot a)f(k^{-1}h).
\end{eqnarray}

\subsection{\bf{The  topology of the dual space of the group $G$.}}$ $\\
The dual space or spectrum of $G$ has been described by G.Mackey (for details, see \cite{Mac1} and \cite{Mac2}).

For each character $\chi\in \widehat A$ and any irreducible unitary representation $\mu$ of the stabilizer $K_\chi$ of $\chi$
in $K$, we have that
\begin{eqnarray}
 \si_{(\mu,\chi)}:=\mu\otimes\chi
\end{eqnarray}
is an irreducible unitary representation of
\begin{eqnarray*}\label{}
 \nn G_\chi:=K_\chi\ltimes A,
 \end{eqnarray*}
whose restriction to $A$ is a multiple of $\chi$ (see \cite{Bag} proposition $2$ p.181)  and the induced representation
$\pi_{(\mu,\chi)}:=\text{ind}_{G_\chi}^G\si_{(\mu,\chi)}$ is an irreducible representation of $G$.

On the other hand, every irreducible unitary representation $\tau_\la$ of $K$ extends to an irreducible representation
(also denoted by $\tau_\la$) of the entire  group $G$ defined by
$$\tau_\la(k,a):=\tau_\la(k),\ \ (k,a)\in G.$$

The following Propositions give  the relationship between $\widehat G$ and the set of all elements  $\pi_{(\mu,\chi)}$ (see \cite{Mac3})
\begin{proposition}\label{prop0.2}
 Let $\pi$ be an irreducible unitary representation of $G$. Then $\pi\res A $ is supported by a $K $-orbit $\th $ in $\widehat A $.
 Suppose $\chi$ is an element of $\theta$. Then $\pi$ is equivalent to a representation of the form $\pi_{(\mu,\chi)}$ where $\mu$
 is an irreducible unitary representation of $K_\chi.$
\end{proposition}
\begin{proposition}\label{prop0.3}
 Let $\pi$ and $\theta$ be as in the above proposition. Assume $\chi$ and $\chi'$ are elements of $\theta.$ i.e, $\chi'=k\cdot\chi$ for some element
 $k$ of $K$. Suppose further that $\pi$ is equivalent to $\pi_{(\mu,\chi)}$ and also is equivalent to $\pi_{(\mu',\chi')}$
 for two elements $\mu\in \widehat{K_\chi},\mu'\in\widehat K_{\chi'} $. Then
\begin{enumerate}
 \item $G_{\chi'}=k\cdot G_\chi\cdot k\inv.$

 \item the representation $h\mapsto \mu'(k hk\inv), h\in K_\chi$,  is equivalent to the representation $\mu$.

\end{enumerate}

\end{proposition}
\begin{definition}\label{def0.2}
 A cataloguing triple we mean a triple $(\chi,J,\mu)$, where $\chi$ is a character of $A,$ $J$ is the stabilizer $K_\chi$ and $\mu$ is an
 irreducible unitary representation of $K_\chi.$ We denote by $\pi_{\chi,J,\mu}:=\pi_{(\mu,\chi)}$ the induced representation
 $\displaystyle\text{ind}_{_{J\ltimes A}}^G(\mu\otimes\chi).$
\end{definition}

By Baggett in \cite{Bag} (Proposition $2.4$-$D$ $p. 187$), we have
\begin{proposition}\label{prop0.4}
 The mapping $(\chi,K_\chi,\mu)\longrightarrow\pi_{(\mu,\chi)}$ is onto $\widehat G.$
\end{proposition}

For the proof of the following theorem, see $\S\S4.\ 1$-$D$ of \cite{Bag}.
\begin{theorem}
The $C^*$-algebra $C^*(G) $ of the locally compact group $G=K\ltimes A$ is  $CCR $.
\end{theorem}
\begin{definition}\label{def0.4}$ $
 Let $F$ be the set of all functions $f$ which satisfy:
 \begin{enumerate}
  \item The domain of $f$ is a closed subgroup of $G$.
  \item $f$ is complex-valued and continuous on its domain.

\end{enumerate}
\end{definition}
In the section $3$ of \cite{Fell-2}, Fell has defined a topology on $F$.
\begin{definition}\label{def0.5}
 The Fell topology on $F$ is the topology defined as follows.
 Let $(f_n)$ be a net of elements of $F$. For each $n$, let
 $H_n$ be the domain of $f_n$. Let $f$ be in $F$ and denote by $H$ the domain of $f$.
 Then the net $f_n$ converges to $f$
 in the Fell topology if and only if the following two conditions hold.
 \begin{enumerate}
  \item The net $H_n$ converges to $H$ in $\K(G).$
  \item For each subnet $(f_{n_m})_m$ of the net $(f_n)_n$ and for each net $h_m$ of
  points of $G$ such that for each $m,$ $h_m$ is an
  element of $H_{n_m}$ and such that the net $(h_m)_m$ converge to an element $h$ of $H$,
  the net of complex numbers $(f_{n_m}(h_m))_m$ converges to
  $f(h),$
 \end{enumerate}
  
\end{definition}
In \cite{Fell-2} Fell describes the topology on $\A(G)$ in terms of the elements of
$F$ and in terms of the Fell topology on $F$. Here is one
consequence of that description.
\begin{theorem}\label{the0.1}
 Let $(H_n,S_n)_n$ be a net of elements of $\A(G)$ and $(H,S)$ be an element of $\A(G)$.
 Assume that, for some function of positive type $f$
  associated with $S$, there exists a net $(f_n)_n$ of functions which satisfies:
  \begin{enumerate}
   \item Each $f_n$ is a finite sum of functions of positive type associated with $S_n$.
   \item The net $(f_n)_n$ converges to $f$ in the Fell topology of $F$.
  \end{enumerate}
Then the net $(H_n,S_n)_n$ converges to $(H,S)$ in $\A(G).$
\end{theorem}

The topology of the dual space of the group $G$ has been described  by Baggett in \cite{Bag} in the following theorem:

\begin{theorem}\label{the0.2}
 The topology of $\widehat G$ may be described as follows: Let $B$ be a subset of $\widehat G$ and $\pi$ an element
 of $\widehat G$. $\pi$ is contained
 in the closure of $B$ if and only if there exist: a cataloguing triple $(\chi,(K_\chi,\mu))$ for $\pi$,
 an element  $(H,S)$
 of $\A(K)$ and a  net $(\chi_n,(K_{\chi_n},\mu_n))_n$ of cataloguing triples, such that:
 \begin{enumerate}
  \item For each $n$, the element $\pi_{(\mu_n,\chi_n)}$ of $\widehat G$ is an element of $B.$
  \item The net $(\chi_n,(K_{\chi_n},\mu_n))$ converges to $(\chi,(K',\mu'))$ in $\widehat A\times\A(K).$
  \item $K_\chi$ contains $K'$, and $\text{ind}_{K'}^{K_\chi} \mu'$ contains $\mu$.
 \end{enumerate}

\end{theorem}

The following lemma is the key for Definition \ref{def0.3.2}.
\begin{lemma}
  Let $(\chi_n,(K_{\chi_n},\mu_n))_n$ be a properly converging net with the  limit  $(\chi,(K',\mu'))$ (i.e $K'$ is a closed subgroup of $K_\chi$ and
 $\mu'\in\widehat {K'}$).
 Then for some subnet we have that $$\dim{\mu_k}=d_{\mu_k}=d_{\mu'}=\dim{\mu'} \text{ for all }k \text{ in the subnet}.$$
\end{lemma}
\begin{proof}
The limit set $L $ of the net $(\pi_k)_k $ in $\wh{  G} $  is according to \cite{Bag} the set
\begin{eqnarray*}\label{}
 \nn L=\{\pi_{(\nu,\chi)}; \nu\in\widehat{K_\chi}, \nu\res{  K'}\ni \mu'\}.
 \end{eqnarray*}
 We can as in \cite{Bag} realize all these  representations $\pi_{(\mu,\chi)} $ on subspaces in the common Hilbert space $L^2(K) $.
Take some $\mu\in\widehat {K_\chi} $   such that $\mu\res{  K'} $ contains $\mu' $.
 Choose a $\rh\in\widehat K $, such that $\rh\res {K_\chi} $ contains $\mu $.
 Let $L^2(K)^\rh$ be the minimal  left and right translation $K $-invariant subspace of $L^2(K) $
 containing $\H_\rh $ and thus also a copy of the Hilbert space $\H_\mu $ of the representation $\mu $ of $K_\chi $.
 Then $\L2{K}^\rh $ has dimension $d_\rh^2 $.

  Then $\L2K^\rh $ also contains a copy $(\H_{\mu'},\mu') $ of the irreducible representation
 $\mu' $ of $K' $. Let $\H^{\mu'} $ be the $\mu' $-isotopic component inside ${\H^\rh}$. We write
 \begin{eqnarray*}\label{}
 \nn \L2K^\rh&= &
  \H^{\mu'}\oplus \H^\mu_0,
 \end{eqnarray*}
where $\H^\rh_0:= (\H^{\mu'})^\perp \subset \L2K^\rh $. There exists $l_\mu\in \N_{>0} $ such that   $\H^{\mu'}\simeq l_\mu \H_{\mu'} $.
Similarly we have $\L2K^\rh=\H^{\mu_n}\oplus \H^{\mu_n}_0 $, $\mu\res{K_{\chi_n}}\simeq l_n \mu_n $, $\H^{\mu_n}\simeq l_n \H_{\mu_n} $
for some $l_n\in\N $. Of course $l_n\leq \dim{\L2K^\mu } $ and $l_n $ can  be $0$.

However, since the representations $(\pi_{(\mu_n,\ch_n)}) $ converge to the representation $\pi_{(\mu,\ch)} $, we can assume   that for a subnet the subspaces $\H^{\mu_n} $ can  be realized inside  $L^2(K)^\rh$ for $n $ large enough (see \cite{Bag}, 4.2-D Theorem).
This tells us also that  $l_n>0 $ for $n $ large  enough and hence, again for a subnet, we can suppose that $l_n=l $ is fixed for every $n $.
Since the dimensions of the spaces $\H^{\mu_n} $ are smaller than the dimension of $\L2K^\rh$, we can also assume that all the dimensions $d_{\mu_n} $ are the same and equal to some common $d>0 $.

We choose for every $n $ an orhonormal basis $(\xi^n_j)_j $ of $\L2K^\rh$,
which passes through the $l_n $ copies of $\H^{\mu_n} $ and through $\H^{\mu_n}_0 $.
We choose the $\xi^n_j $ such that $\text{span}\{ \xi^n_1,\cdots, \xi^n_{d_{\mu_n}}\} $ is  a copy of the space $\H_{\mu_n} $.
Since the dimension of $\L2K^\rh$ is finite, we can assume (passing to a subnet)  that $\limn \xi_j^n=\xi_j $ exists in $\L2K $
for every $j \leq ld$.
Let $c_n, n\in N, $ be the character of the irreducible representation $\mu_n $ of $K' $.
Then
  $c_n\ast \xi_n=\xi_n $ for any $\xi_n\in \H^{\mu_n} $ and $c_n\ast H^{\mu_n}_0=\{ 0\} $ for every $n $.
 According to Bagget, the pairs $(K_{\chi_n},c_{n}) $ converge in $\A(K) $ to the pair $(K',c_{\mu'}) $,
 where $c_\nu $ denotes  the character of an irreducible representation $\nu $ (see \cite{Bag}, 7.1-B Lemma). This implies  that (see \cite{Bag},1.4-A Proposition)
\begin{eqnarray*}\label{}
 \nn \xi_j&:=&\limn \xi^n_j\\
 \nn  &= &
 \limn {c_n} \ast \xi_j^n\\
 \nn  &= &
c_{\mu'}\ast \ch_j
 \end{eqnarray*}
for every $j$ and similarly
\begin{eqnarray*}\label{}
 \nn c_{\mu'}\ast \H^\mu_0 &= &
 \{ 0\}.
 \end{eqnarray*}
 This shows that $K' $ acts on $\H^{\mu'}=\limn \H^{\mu_n} $ by a multiple of $\mu' $.
Define for any   $n $ and $i,j \in \{ 1,\cdots, d\}$ the function $c^n_{i,j} $ on $K $ by
\begin{eqnarray*}\label{}
 \nn c^n_{i,j}(k) &:= &
 \frac{1}{d}\langle \la(k)\xi^n_j,\xi^n_i\rangle_{\L2K}, k\in K,
 \end{eqnarray*}
where $\la $ denotes the left regular representation of $K $ on $L^2(K) $. Since $\limn \xi^n_i=\xi_i $ in $\L2K $ for every $i $,
the functions $c^n_{i,j} $ converge uniformly on $K $ to
the function $c_{i,j}$, where
\begin{eqnarray*}\label{}
 \nn c_{i,j}(k)&:=&\frac{1}{d}\langle \la(k)\xi_j,\xi_i\rangle_{\L2K}.
 \end{eqnarray*}
Now  the operator ${\mu_n(c^n_{i,j}}\res{  K_{\ch_n}}) $,
which acts on the Hilbert space $\H_{\mu_n}:=\text{span}\{ \xi^n_j, 1\leq j\leq d,\} $ is the rank one operator
$P_{\xi^n_i,\xi^n_j} $ which converges  to the operator $P_{\xi_i,\xi_j} $ on the Hilbert space $\H_{\mu'}:=\limn \H_{\mu_n} $. This shows that the restriction of $\mu'$ to $\H_{\mu'} $ is irreducible. Hence $d=\dim {\H_{\mu'}} $ and $l_{\mu'}=l $.
\end{proof}

\subsection{}\label{CCR}
\begin{remark}\label{operator convergence for the mun}
\rm   Identifying for every $n $ the Hilbert space $\H_{\mu_n} $ with $\C^d $ via the basis given by the $\xi_j $'s,
we see also that for every $f\in C(K) $ the operators $\mu_n(f\res{  K_{\ch_n}})  $ converge strongly and hence in operator norm to
the operator $\mu'(f\res{  K'}) $.
 \end{remark}

We have by  $\S\S 4.5$ of \cite{Dix}:

\begin{theorem}
 Let $A$ be a postliminal $C^*$-algebra. Then $A$ admits a
composition net $(I_n)_{0\leq n\leq\al}$ such that 
for any $n $,  which is not an ordinal,  the quotient $I_{n+1}/I_{n} $ is with continuous trace and such that
for every ordinal $\be\leq \al $ the relation  $I_\be=\bigcup_{n<\be}I_n $ holds.  
\end{theorem}
We take now as $C^* $-algebra our $A=C^*(G) $, which is $CCR $.
 
Let  
\begin{eqnarray*}\label{}
 \nn S_n&:= &
 \{ \pi\in\widehat G\vert \pi(I_n)=\{ 0\}\}, 0\leq n\leq \al.
 \end{eqnarray*}
Then $S_0=\widehat G $ and $S_\al=\{ \ey\} $. The subsets
\begin{eqnarray*}\label{}
 \nn \GA_n &:= &
 S_{n}\setminus S_{n+1}, 0\leq n\leq\al, n\text{ not an ordinal},
 \end{eqnarray*}
are locally compact and Hausdorff in their relative topologies, since $\GA_n $ is the spectrum of the algebra $I_{n+1}/I_{n} $,
which is of continuous trace (see \cite{Dix}).
Let  $S=\widehat{G} $ be  the spectrum of $G $. Then
\begin{eqnarray*}\label{}
 \nn S &= &
 \bigcup_{0\leq n\leq\al} S_n\\
 \nn  &= &
\dot\bigcup_{0\leq n\leq\al} \GA_n.
 \end{eqnarray*}

\subsection{The Fourier transform.}
Let us first write down explicitly the representation $\pi_{(\mu,\chi)}$. Its Hilbert space $\H_{{(\mu,\chi)}}$ can be identified
with the space
\begin{eqnarray*}
 L^2(G/K_\chi\ltimes A,\si_{(\mu,\chi)})\simeq L^2(K/K_\chi,\mu)\subset L^2(K,\H_\mu).
\end{eqnarray*}
Let $\xi$ be an element of $\H_{{(\mu,\chi)}}.$ For all $a\in A$ and $k,h\in K$ we use the same calculation as in (\ref{calcdelta})  and we have that
\begin{eqnarray*}
 \pi_{(\mu,\chi)}(k,a)\xi(h)&=&\chi(1_K,h^{-1}\cdot a)\xi(k^{-1}h)\\
 &=:&\chi(h^{-1}\cdot a)\xi(k^{-1}h).
\end{eqnarray*}
Let us compute for $f\in L^1(G)$ the operator $\pi_{(\mu,\chi)}(f).$ We have for $h\in K$ and $\xi\in\H_{{(\mu,\chi)}}$ that
\begin{eqnarray}\label{calcoppimuchi}
 \nn\pi_{(\mu,\chi)}(f)\xi(h)&=&\int_Gf(k,a)\pi_{(\mu,\chi)}(k,a)\xi(h)dadk\\
 \nn&=& \int_K\int_Af(k,a)\chi(h^{-1}\cdot a)\xi(k^{-1}h)dadk\\
  \nn&=& \int_K\int_Af(hk^{-1},a)\chi(h^{-1}\cdot a)\xi(k)dadk\\
 \nn&=&\int_{K/K_\chi}\int_{K_\chi}(\int_Af(hs^{-1}k^{-1},h\cdot a)\chi( a)da)\mu(s^{-1})\xi(k)ds dk\\
 \nn&=&\int_{K/K_\chi}\int_{K_\chi}\widehat f^2(hs k^{-1},h\cdot \chi)\mu(s)\xi(k)ds dk\\
 &=&\int_{K/K_\chi}f_{\mu,\chi}(h,k)\xi(k)dk,
\end{eqnarray}
where
\begin{eqnarray}\label{fmuchi}
\nn f_{\mu,\chi}:K\times K &\longrightarrow& \mathcal{B}(\mathcal{H}_{\mu}) \\
(h,k) &\longmapsto& \int_{K_\chi}\widehat f^2(hs k^{-1},h\cdot\chi)\mu(s)ds .
\end{eqnarray}
and
\begin{eqnarray*}\label{}
 \nn \hat f^2(k,\mu) &:= &
 \int_A \ch (a) f(k, a)da, k\in K,\ch\in \hat A.
 \end{eqnarray*}

\begin{definition}\label{def0.6}
For each $f\in C^*(G),$ the Fourier transform $\mathcal{F}(f)$ of $f$
is the isometric homomorphism on $C^*(G)$ into $\ell^\infty(\widehat{G})$ which is given by
\begin{eqnarray*}
  \mathcal{F}(f)(\mu,\chi) &=& \pi_{(\mu,\chi)}(f)\in\mathcal{B}(\mathcal{H}_{(\mu,\chi)}),\,\, (\chi,K_\chi,\mu)\text{ is a cataloguing triple}.
\end{eqnarray*}
\end{definition}
Let now $L^1(G)_c$ be the dense subspace of $\L1G $ defined by
\begin{eqnarray*}\label{}
 \nn\L1G_c&:=&\left\{f\in \L1G; \text{ the function }\widehat f^2 \text{ is in }C_c(K\times\widehat A).\right\}.
 \end{eqnarray*}

 \begin{definition}\label{Pnuk def}$ $
 Let $L $ be  a closed subgroup of the compact group $K $ and let $(\nu,\H_\nu) $ be an irreducible representation of $L $ with character $\ch_\nu $. We may identify the Hilbert space $\H_\nu $ with $\C^d $, $d=d_\nu=\dim\nu $.
 Let
 \begin{eqnarray*}\label{}
 \nn \L2\nu &:= &
 \L2{  K,\H_\nu}\simeq \L2{  K,\C^d}.
 \end{eqnarray*}

 \begin{enumerate}
  \item Define for $\chi\in \widehat A $ such that $L\subset K_\ch $ and  for $f\in \L1G $ the operator
$\ta_{\nu,\chi}(f) $ on $\L2{K,\C^d} $ by
\begin{eqnarray*}\label{}
 \nn \ta_{\nu,\chi}(f)\xi(x) &:= &
 \int_{K} \Big(\int_{L}(\widehat f^2(x l y\inv, y\cdot \chi))\nu(l)dl\Big)(\xi(y)) dy.
 \end{eqnarray*}
 \item Define for  $\chi\in \widehat A $, for a closed subgroup $L \subset K_\chi$ of $K $, for $\nu \in\widehat L$,
the linear projection $P_{\nu}:\L2{  K,\H_\nu}\to \L2{  K/L,\nu} $ by
\begin{eqnarray*}\label{}
 \nn P_{\nu}(\va)(x) &:= &
 \int_{L}\nu(l)(\va(xl))dl, \va\in C(K,\H_\nu), x\in K.
 \end{eqnarray*}

 \end{enumerate}

\begin{proposition}$ $
\begin{enumerate}
 \item The linear operator $P_{\nu} $ is an selfadjoint  projection of the Hilbert space $\L2\nu $.
 \item For any closed subgroup $L\subset K_\chi $ of $K $, $\nu\in \widehat L $, $\chi\in \widehat A $ and $f\in \L1G $ we have
 \begin{eqnarray*}\label{}
 \nn \text{ind}_{  L\times A} ^G \nu\otimes \chi (f)\circ P_{\nu} &= &
\ta_{\nu,\chi}(f), f\in \L1G.
 \end{eqnarray*}
\end{enumerate}

\end{proposition}
\begin{proof}
\begin{enumerate}
 \item We have for $\va\in \L2{  K,\H_\nu} $ that
 \begin{eqnarray*}\label{}
 \nn \no{P_\nu(\va)}^2 &= &
 \int_{K/L} \left\Vert{\int_L \nu(l)(\va(kl))dl}\right\Vert_{ \C^d} ^2 d\dot k\\
 \nn  &\leq &
\int_{K/L}\left( {\int_L \left\Vert\va(kl)\right\Vert _{ \C^d}^2 dl}\right) d\dot k\\
 \nn  &=&
\int_{K} \Vert\va(k)\Vert _{ \C^d}^2  dk\\
\nn  &=&
\no{  \va}^2  .
 \end{eqnarray*}
Let $\va\in \H_{\nu,\ch}=\L2{  K/L,\nu}\subset \L2{  K,\H_\nu} $.
Then for $k\in K $,
\begin{eqnarray*}\label{}
 \nn P_\nu(\va)(k) &= &
 \int_L \nu(l)(\va(kl))dl\\
 \nn  &= &
  \int_L \nu(l)\nu(l)\inv(\va(k))dl\\
  \nn  &= &
\int_L \va(k)dl\\
\nn  &= &
\va(k).
 \end{eqnarray*}
Hence the operator $P_\nu $ is the identity on $\H_{\nu}\subset \L2{  K,\H_\nu} $.

Let $\mu\in\widehat L $ and let $c_\mu $ its character. For $\va\in \L2{K,\H_\nu}$ let
\begin{eqnarray*}\label{}
 \nn \va_\mu(k) &:= &\va\ast c_\mu(k)\\
 \nn  & = &
 \int_L \va(kl\inv)c_\mu(l)dl, k\in K.
 \end{eqnarray*}
Then the mapping $\va_\mu $ is also contained in $ L^2(\nu)$  and for another $\om\in\widehat L $ we have that
\begin{eqnarray*}\label{}
 \langle \va_\mu,\va_\om\rangle _{\L2 \nu}
 \nn  &= &
\int_{K}\langle \va_\mu(k),\va_\om(k)\rangle _{\H_\nu}dk\\
 \nn  &= &
\int_{K/L}\int_L\langle \va_\mu(kl),\va_\om(kl)\rangle _{\H_\nu}dld\dot k\\
\nn  &= &
\int_{K/L}\int_L\int_L\langle \va(kll_1\inv)c_\mu(l_1)dl_1,\int_L\va(kll_2\inv)c_\om(l_2)dl_2\rangle _{\H_\nu}dld\dot k\\
\nn  &= &
  \int_{K/L}\int_L\int_L\langle \va(kl_1\inv)c_\mu(l_1l)dl_1,\int_L\va(kl_2\inv)c_\om(l_2l)dl_2\rangle _{\H_\nu}dld\dot k\\
  \nn  &= &
  \int_{K/L}\int_L\int_L\langle \va(kl_1\inv),\va(kl_2\inv)\rangle_{\H_\nu} dl_1dl_2\int_Lc_\mu(l_1l)\ol{c_\om(l_2l)}dl d\dot k
\\
\nn  &= &
\int_{K/L}\int_L\int_L\langle \va(kl_1\inv),\va(kl_2\inv)\rangle_{\H_\nu} dl_1dl_2\ 0\ d\dot k
\\
\nn  &= &
0.
 \end{eqnarray*}
 It is easy to see now that
 \begin{eqnarray*}
 \L2{\nu}=\sum_{\mu\in \widehat{L}
}\L2\nu\ast c_\mu,
 \end{eqnarray*}
since $\L2K=\sum_{\mu\in\hat{L}}\L2K\ast c_\mu
 $.
Furthermore, for $\mu\ne \nu $ it follows that
\begin{eqnarray*}\label{}
 \nn P_\nu(\va_\mu)(k) &= &
 \int_L \nu(l)(\va_\mu(kl))dl\\
 \nn  &= &
\int_L \nu(l)\left(\int_L\va_\mu(kll_1\inv)c_\mu(l_1)dl_1\right)dl\\
\nn  &= &
\int_L \nu(l)\left(\int_L\va_\mu(kl_1\inv)c_\mu(l_1l)dl_1\right)dl\\
& =&
\int_L \nu(l)c_\mu(l_1l)dl\left(\int_L\va_\mu(kl_1\inv)dl_1\right)\\
& =&
0\cdot \left(\int_L\va_\mu(kl_1\inv)dl_1\right)\\
\nn  &= &
0.
 \end{eqnarray*}
Hence $\L2\nu^\perp =\sum_{\mu\ne \nu}\L2\nu_\mu $ and $P_\nu $ is zero on $\L2\nu^\perp $. This shows that
\begin{eqnarray*}
 P_\nu^*
 =P_\nu.
 \end{eqnarray*}

\item For $f\in C_c(G), \va\in \L2{  K,\C^d}, x\in K $, we have by (\ref{calcoppimuchi}) that
\begin{eqnarray*}\label{}
 \nn  & &
 \text{ind}_{  L\times A} ^G \nu\otimes \ch (f)(P_{\nu}(\va))(x)\\
 \nn  &= &
\int_{K/L}\int_L \widehat f^2(xl k\inv,k\cdot \chi)\nu(l)dl(P_\nu(\va)(k)) dk\\
 \nn  &= &
\int_{K}\int_L \widehat f^2(xl k\inv,k\cdot \chi)\nu(l)dl(P_\nu(\va)(k)) dk
\end{eqnarray*}
\begin{eqnarray*}
\nn  &\overset{L\subset K_\chi}= &
\int_{K}\int_L \widehat f^2(xl k\inv,k\cdot \chi)\nu(l)dl(\int_L\nu(l')(\va)(kl')) dl'dk\\
\nn  &= &
\int_{K}\int_L \widehat f^2(xl l'k\inv,k\cdot \chi)\nu(l)dl(\int_L\nu(l')(\va)(k)) dl'dk\\
\nn  &= &
\int_{K}\int_L \widehat f^2(xl l'k\inv,k\cdot \chi)dl(\int_L\nu(ll')(\va)(k)) dl'dk\\
 &= &
\int_{K}\int_L \widehat f^2(xlk\inv,k\cdot \chi)dl(\int_L\nu(l)\va)(k) dl'dk
\\
\nn  &= &
\int_{K}\int_L \widehat f^2(xlk\inv,k\cdot \chi)\nu(l)(\va)(k)dl dk
\\
\nn  &= &
\ta_{\nu,\chi}(f)(\va)(x).
 \end{eqnarray*}
 \end{enumerate}
\end{proof}

\end{definition}

\begin{lemma}\label{lem0.3}
Let $f\in C^*(G).$ Then we have:
\begin{enumerate}
  \item $\underset{(\mu,\ch)\longrightarrow \infty}{\lim}\|\mathcal{F}(f)(\mu,\chi)\|_{op}=0$ .
  \item

Let $(\chi_m, L_m,\mu_m)_{m\in M} $
be  a converging net in $\widehat A\times\A(K)  $ with limit $(\chi_\iy,L_\iy,\mu_\iy) $ such that $\dim{  \mu_m}=\dim{\mu_\iy} $ for any $m $.
Then, identifying the Hilbert spaces of the representations $\mu_m $ with $\C^d $,
we have that $\limm \ta_{\mu_m,\chi_m}(f)=\ta_{\mu_\iy,\chi_\iy}(f) $ in operator norm for any $f\in\L1G $.

\end{enumerate}
\end{lemma}
\begin{proof}
 \begin{enumerate}
\item Let $A $ be a  $C^* $-algebra. According to (\cite{Dix}, chapter 3\ \S3.3) ,  
if a net $(\pi_k)_k \subset \widehat A$ goes to infinity, i.e. this net has no converging subnet, 
then $\limk \noop {\pi_k(a)}=0 $ for any $a\in A. $
Now Bagget \cite{Bag} has shown that for every net of cataloguing triples $(\ch_k, K_{\ch_k},\mu_k)_k $ we have that $(\mu_k,\ch_k)_k $
goes to infinity, if and only if the net $(\pi_{(\mu_k,\ch_k)})_k $ goes to infinity in $\widehat{C^*(G)} $.
\item Let first $f $ be contained in $\L1G_c $. Let
 \begin{eqnarray*}\label{}
 \nn f_{m}(x,y) &:= &
 \int_{L_m} \widehat f^2(xly\inv,y\cdot \ch_m)\mu_m(l)dl, x,y\in K, m\in M\cup \{ \iy\}.
 \end{eqnarray*}
Then we see that $f_{m}(x,y)\in \B(\C^n) $ and by Remark that \ref{operator convergence for the mun}
 \begin{eqnarray*}\label{}
 \nn \limm f_m(x,y)v & =&f_\iy(x,y)v, v\in\C^d,
 \end{eqnarray*}
point wise in $x,y $. Therefore also

 \begin{eqnarray*}\label{}
 \nn M_d(\C)\ni \limm f_m(x,y) & =&f_\iy(x,y),
 \end{eqnarray*}
 where $M_d(\C)$ denotes the  space of complex matrices of size $d$.

Using Lebesgue, we see that
\begin{eqnarray*}\label{}
 & &\limm \no{  \ta_{\mu_m,\chi_m}(f)-\ta_{\mu_\iy,\chi_\iy}(f)}_{H.S}^2\\
 \nn&=&  \limm \int_{K\times K}\no{  f_m(x,y)-f_\iy(x,y)}_{H.S}^2dxdy\\
 &=&
 0.
 \end{eqnarray*}

Hence
\begin{eqnarray*}\label{}
\limm \ta_{\mu_m,\chi_m}(f)=\ta_{\mu_\iy,\chi_\iy}(f).
 \end{eqnarray*}
 The lemma follows now from the density of $L^1(G)_c$ in $C^*(G)$.
 \end{enumerate}

\end{proof}
\subsection{\bf{A $C^*$-condition.}}
Let $G=K\ltimes A $ be as before a semi-direct product of a compact group $K $ with a locally compact abelian group $A $.
\begin{remark}\label{a converging net}
\rm   Let
$(\pi_{(\mu_m,\chi_m)})_{m\in M} $ be a net in $\widehat G $ which converges to $\pi_{(\mu_\iy,\chi_\iy)} $. We can suppose that for  a subnet
(also denoted by $M $ for simplicity of notation) that the triples $(\chi_m, K_{\chi_m}, \mu_m)$
converge to $(\chi_\iy,K_\iy,\mu_\iy) $ in $\widehat A\times \A(K)  $ and that the Hilbert spaces $\H_{\mu_m} $ and $\H_{\mu_\iy} $
are identified with $\C^d $ for some $d\in\N^* $ and that all these spaces
$\H_{(\mu_m,\ch_m)}, m\in M\cup\{ \iy\}$, are subspaces of the common
Hilbert space $\H_M:=\L2{  K,\C^d} $.
The representation $\ta_{\mu_\iy,\ch_\iy}=\text{ind}_{K_\iy\times A}^G \mu_\iy\otimes\ch_\iy $ can be disintegrated
into an integral of irreducible
representations supported by the limit set
\begin{eqnarray*}\label{}
 \nn L &=&\left\{ \pi_{(\mu,\chi)}\vert\ \mu\in \widehat{K_{\ch_\iy}},\ \mu\res{K_\iy}\text{ contains }\mu_\iy\right\}
 \end{eqnarray*}
of the net $(\pi_{(\mu_m,\ch_m)})_m $ (see Theorem \ref{the0.2}). We denote by $\si_{\mu_\iy, \chi_\iy} $ the corresponding
representation of the algebra $\widehat{C^*(G)}\res L $ on the Hilbert space $\L2{  K/K_\iy,\mu_\iy}\subset \L2{  K,\C^d} $.
Let us observe that by the construction of $ \si_{\mu_\iy, \chi_\iy}$ we have that
\begin{eqnarray*}\label{}
 \nn \si_{\mu_\iy, \chi_\iy}(\widehat a\res L) &= &\ta_{\ch_\iy,\mu_\iy}(a), a\in C^*(G).
 \end{eqnarray*}

We can extend this representation $\si_{\mu_\iy,\ch_\iy} $ to the larger $C^* $-algebra $CB(L) $ consisting of all uniformly bounded  operator fields $ F $
satisfying $F(\pi)\in \K(\H_\pi),\pi\in L, $ and we denote this extension also by $\si_{\mu_\iy,\chi_\iy} $ (see \cite{Ar69}) .
 \end{remark}

\begin{definition}\label{def0.3.2}
 Let $\D(G)=\D$ be the family consisting of all uniformly bounded  operator fields $F\in\ell^\infty(\widehat{G})$ satisfying the following conditions:
 \begin{enumerate}
  \item $F(\pi)$ is a compact operator on $\H_{\pi}$ for every $\pi\in\widehat G.$
  \item $\underset{(\mu,\ch)\to\infty}{\lim}\noop{F(\mu,\chi)}=0.$
  \item  Let $(\pi_{(\mu_m,\ch_m)})_{m\in M} $ be  a properly converging net
  in $\widehat G $ with the properties and notations of the preceding Remark \ref{a converging net}. Then
  \begin{eqnarray*}\label{}
 \nn \limm \noop{  F(\mu_m,\ch_m)\circ P_{\mu_m}-\si_{\mu_\iy,\ch_\iy}(F\res L)\circ  P_{\mu_\infty}} &= &
0.
 \end{eqnarray*}

 \end{enumerate}

\end{definition}

\begin{proposition}\label{prop0.3.1}
 $\D(G)$ is a $C^*$-algebra for the norm $\noop{\cdot}$ containing $\wh {C^*(G) }$.
\end{proposition}
\begin{proof}
 First we show that $\D$ is a  norm closed involutive subspace of $l^{\iy}(\widehat G) $. It is clear that $\D$ is a sub-space of $\ell^{\infty}(\widehat{G})$. The conditions
 $(1),\ (2)$  are evidently true for every $F$ in the closure $\ol\D $ of $\D $. For the condition $(3)$, let $F\in\ol\D $ and let $(F^k)_k\subset\D$ such that
 $\underset{k\to\infty}{\lim}\no{F^k-F}_\infty=0.$ Then also
 $\limk\no{(F^k)^*-F^*}_\iy=0 $.
 Hence
 for any $\varepsilon>0$ there exists $k_0$ such that such that $\no{F-F^k}_\infty<\varepsilon$ for any $k\geq k_0.$
  Therefore choosing some $k>k_0 $ we have for $m $ large enough that
  \begin{eqnarray*}\label{}
 \nn \noop{  F^k(\mu_m,\ch_m)\circ P_{\mu_m}-\si_{\mu_\iy,\ch_\iy}({F^k}\res L)\circ  P_{\mu_\infty}}\nn  &\leq &
\ve\\
 \end{eqnarray*}
 and so
  \begin{eqnarray*}\label{}
 \nn  & &
  \noop{  F(\mu_m,\ch_m)\circ P_{\mu_m}-\si_{\mu_\iy,\ch_\iy}(F\res L)\circ  P_{\mu_\infty}} \\
  &\leq &
 \noop{  F(\mu_m,\ch_m)\circ P_{\mu_m}-F^k(\mu_m,\ch_m)\circ P_{\mu_m}}\\
 \nn  &+ &
\noop{  F^k(\mu_m,\ch_m)\circ P_{\mu_m}-\si_{\mu_\iy,\ch_\iy}({F^k}\res L)\circ  P_{\mu_\infty}} \\
\nn  &+ &
\noop{  \si_{\mu_\iy,\ch_\iy}(F^k\res L)\circ  P_{\mu_\infty}-\si_{\mu_\iy,\ch_\iy}(F\res L)\circ  P_{\mu_\infty}} \\
 \nn  &\leq &
  \ve+\ve +
  \noop{ P_{\mu_\iy}\circ   \si_{\mu_\iy,\ch_\iy}({F^k}\res L)^*-P_{\mu_\iy}\circ\si_{\mu_\iy,\ch_\iy}(F\res L)^*}\\
   \nn  &\leq &
  2\ve+
  \noop{    \si_{\mu_\iy,\ch_\iy}({(F^k)^*}\res L-{F^*}\res L)}\\
  \nn  & \leq &
3\ve.
 \end{eqnarray*}
Hence  $F\in\D$. Since $\si_{\mu_\iy,\ch_\iy} $ is  a representation, it follows that $\D $ is involutive and so
 $\D $ is an involutive  Banach space. Let us show that it is an algebra.

Let $F,F'\in\D $. We must show that $F\circ   F' $ is in $\D $ too.

  The conditions $(1), (2) $ are necessarily true for $F\circ   F' $.

Let us check point (3). It follows from property $(3) $ for $F $ , using the involution $^* $,   that also
\begin{eqnarray*}\label{}
 \nn  \limm \noop{ P_{\mu_k}\circ   F(\mu_m,\ch_m) -P_{\mu_\infty}\circ  \si_{\mu_\iy,\ch_\iy}(F\res L) }&= &
 0.
 \end{eqnarray*}

We then have that,  since $P_\mu \circ F(\mu,\ch)\circ P_\mu= F(\mu,\ch)\circ P_\mu$ for every $\ch\in\widehat{A}, \mu\in\widehat{K_\ch}
$,
 \begin{eqnarray*}\label{}
 \nn \nn  & &
  \limm \noop{  F\circ   F'(\mu_m,\ch_m)\circ P_{\mu_m}-\si_{\mu_\iy,\ch_\iy}(F\circ  F'\res L)\circ  P_{\mu_\infty}}\\
  &= &
 \nn \limm \noop{ P_{\mu_m}\circ   F(\mu_m,\ch_m)\circ   F'(\mu_m,\ch_m)\circ P_{\mu_m}-
 P_{\mu_\iy}\circ  \si_{\mu_\iy,\ch_\iy}(F\res L)\circ \si_{\mu_\iy,\ch_\iy}(F'\res L) \circ  P_{\mu_\infty}}\\
  &\leq &
 \nn \limm \noop{ P_{\mu_m}\circ   F(\mu_m,\ch_m)\circ   F'(\mu_m,\ch_m)\circ P_{\mu_m}-
 P_{\mu_\iy}\circ  \si_{\mu_\iy,\ch_\iy}(F\res L)\circ F'(\mu_m,\ch_m)\circ P_{\mu_m}}+\\
  &+ &
 \nn \limm \noop{ P_{\mu_\iy}\circ  \si_{\mu_\iy,\ch_\iy}(F\res L)\circ F'(\mu_m,\ch_m)\circ P_{\mu_m}-
 P_{\mu_\iy}\circ  \si_{\mu_\iy,\ch_\iy}(F\res L)\circ  \si_{\mu_\iy,\ch_\iy}(F'\res L) \circ  P_{\mu_\infty}}\\
 \nn  &\leq &
 \nn \limm C\noop{ P_{\mu_m}\circ   F(\mu_m,\ch_m)-P_{\mu_\iy}\circ  \si_{\mu_\iy,\ch_\iy}(F\res L)}+\\
  &+ &
 \nn \limm C\noop{  F'(\mu_m,\ch_m)\circ P_{\mu_m}- \si_{\mu_\iy,\ch_\iy}(F'\res L) \circ  P_{\mu_\infty}}\\
 \nn  &= &
0,
 \end{eqnarray*}
where $C=\max(\no F_\iy,\no {F'}_\iy )$.

Since $\widehat{C^*(G)}
 $ satisfies all the conditions of $D^*(G) $ it follows that $\widehat{C^*(G)}
 $ is contained in $D^*(G) $.
\end{proof}

 \begin{proposition}\label{spectre D}
The spectrum $\widehat{\D(G)} $ of the algebra $\D(G) $ can be identified with $\widehat G. $
 \end{proposition}
 \begin{proof}
 We have by  $\S\S 4.5$ of \cite{Dix}:

\begin{theorem}
 Let $A$ be a postliminal $C^*$-algebra. Then $A$ admits a
composition sequence $(I_n)_{0\leq n\leq\al}$ such that the quotients $I_{n+1}/I_n$ are $C^*$-
algebras with continuous trace.
\end{theorem}
This theorem applies of course to our $C^* $-algebra $C^*(G) $. Let now
 \begin{eqnarray*}\label{}
 \nn S_n&:= &
 \{ \pi\in\widehat G\vert \pi(I_n)=\{ 0\}\}, 0\leq n\leq \al.
 \end{eqnarray*}
The subsets
\begin{eqnarray*}\label{}
 \nn \GA_n &:= &
 S_{n}\setminus S_{n+1}, 0\leq n\leq\al,
 \end{eqnarray*}
are locally compact and Hausdorff in their relative topologies, since $\GA_n $ is the spectrum of the algebra $I_{n+1}/I_{n} $, 
which is of continuous trace (see \cite{Dix}). Then
\begin{eqnarray*}\label{}
 \nn \widehat G &= &
 \bigcup_{0\leq n\leq \al} S_n,\\
 \nn  &= &
\dot\bigcup_{0\leq n\leq\al} \GA_n,\\ 
 \nn
 S_{n-1}&\supset & S_n, 0<n\leq \al\\
 S_0&=&\widehat G,\\
S_\al&=&\{ \ey\}.
 \end{eqnarray*}

 Evidently $\widehat{  \D}\supset \widehat G $.
Define:
\begin{eqnarray*}\label{}
 \nn J_n &:= &
 \{ F\in \D\vert   F(\pi)=0, \pi\in S_{n}\},\ 0\leq n\leq\al.
 \end{eqnarray*}
Then the $J_n $'s are closed ideals of $ \D$ and
\begin{eqnarray*}\label{}
 \nn J_n&\supset &J_{n-1} , 0<n\in N, n\text{ not an ordinal},\\
 \nn  J_{\al}&= &
\D,\\
\nn  J_0&= &\{0\}.
 \end{eqnarray*}

Let now $\pi\in\widehat \D\setminus \widehat G $. Let
\begin{eqnarray*}\label{}
 \nn n_{\pi} &:= &
 \sup_{n\in N}\pi(J_{n})= \{0\}.
 \end{eqnarray*}
If $n_\pi=\al $, then $\pi(J_{n})=\{0\} $ for every $0\leq n<\al $ and then  $\pi(\D)=\pi(J_0)=\{0\}$,
 which is impossible.
Hence
\begin{eqnarray*}\label{}
 \nn n_\pi<\al.
 \end{eqnarray*}
We have now that $\pi(J_{n_\pi+1})\ne \{ 0\} $, but $\pi(J_{n_\pi})=\{ 0\} $.

This means in particular that $\pi $ is contained in the hull of the ideal $J_{n_\pi} $.  But $J_{n_\pi} $ is the kernel of the subset $S_{n_\pi} $. In other words $\pi $ is an element of the closure of the subset ${  S_{n_\pi}} $, but $\pi\not \in \ol{S_{{n_\pi+1}}} $.

There exists therefore a net $(\pi_k)_k\subset \GA_{n_\pi}$, such that $\pi=\limk \pi_k $ in $\widehat \D $.  Hence, there exists  for $\xi\in \H_\pi $ and for any $k $ an element $\xi_k\in\H_{\pi_k} $, such that for any $F\in\D $, we have that
\begin{eqnarray}\label{lim is pi}
 \nn \limk \langle F(\pi_k)(\xi_k),\xi_k\rangle  &= &
 \langle \pi(F)\xi,\xi\rangle .
 \end{eqnarray}
Then $\pi_k $ does not go to infinity in $\widehat G $ because of condition $(2) $.
Hence, either for a subnet, the net $(\pi_k=\pi_{(\mu_k,\ch_k)})_k $ converges in $\GA_{n_\pi} $ to some 
$\pi_\iy=\pi_{(\mu_\iy,\ch_\iy)}\in \GA_{n_\pi} $, and then by condition (3)
\begin{eqnarray*}\label{}
 \nn \Vert \si_{\mu_\iy,\ch_\iy}(F)\circ  P_{\mu_\iy}- F(\mu_k,\ch_k)\circ  P_{\mu_k}\Vert_{\rm op}&= &
 0
 \end{eqnarray*}
for any $F\in\D $. Now for any $k $ we have that
\begin{eqnarray*}\label{}
 \nn \xi_k=P_{\mu_k}(\xi_k) .
\end{eqnarray*}
This shows that
\begin{eqnarray}\label{si is pi}
 \limk \langle \si_{\mu_\iy,\ch_\iy}(F\res L)\circ  P_{\mu_\iy}(\xi_k),P_{\mu_\iy}(\xi_k)\rangle  &= &
 \langle \pi(F)\xi,\xi\rangle , F\in\D.
 \end{eqnarray}
Since in our case $L=\{ \pi_\iy\} $, relation \ref{si is pi} implies that $\ker{ \pi_\iy}\subset \ker \pi $. But then $\pi=\pi_\iy $, since $\pi_\iy $ is completely continuous (see \cite{Dix}, 4.1.11. Corollary).

The other possibility is that the net converges in $\widehat{C^*(G)} $ to a limit set $L $ contained in $ S_{n_\pi+1} $.

Now for   $F\in J_{n_\pi+1} $, it follows from condition $(3) $, that
\begin{eqnarray*}\label{}
\nn \limm \noop{  F(\pi_k)\circ P_{\mu_k}} &= &
 \limm \noop{  F(\pi_k)\circ P_{\mu_k}-\si_{\mu_\iy,\ch_\iy}(F\res L)\circ  P_{\mu_\infty}} \\
 &= &
 0,
 \end{eqnarray*}
which shows that $\limk \noop {  F(\pi_k)}=0 $. But this implies  then by (\ref{lim is pi}) that  $\pi(F)(\xi)=0 $ for every $\xi $. 
Therefore $\pi $ is $0 $ on $J_{n_\pi+1} $.  This contradiction shows that  $\pi=\pi_\iy\in\widehat G $.
\end{proof}

\begin{theorem}\label{dis hat CG}
Let $G=K\ltimes A $ be the semi-direct product of a compact group $K $ with an abelian locally compact group $A $. Then the $C^* $-algebra $\D(G) $ is isomorphic to the group $C^* $-algebra $C^*(G) $.
 \end{theorem}
\begin{proof} We know now that $\widehat{\D}=\widehat G $.
 It suffices to apply the theorem of  Stone-Weierstrass to the $C^* $-algebra $\D$ and its subalgebra $\widehat{C^*(G)}$.

\end{proof}
\section{\bf{Examples}}

\begin{example} 

\rm
For all $n\in\N^*$, let $G_n$ the the   the semi-direct product of the compact Lie group $SO(n)$ with the
abelian group $\R^n$. The $C^*$-algebra of this group is discribed by Abdelmoula, Elloumi and Ludwig in \cite{Lud-Ell-Abd}.

We can parameterize the dual space in the following way:
\begin{eqnarray*}
 \Gamma_1&=&\widehat{SO(n-1)}\ti\R_+^*\\
 \Gamma_0&=&\widehat{SO(n)}.
\end{eqnarray*}

For $\rho_\mu\in\widehat{SO(n-1)}, $ and $r\in\R_+^*$, we denote by $K_{\chi_r}=SO(n-1)\ti\R_+^*$ the stabilizer of $\chi_r$.  The projection 
$P_r:L^2(SO(n))\longrightarrow L^2(SO(n))$ are 
$$P_r(\varphi)(x):=\int_{K_{\chi_r}}\rho_\mu(l)\varphi(xl)dl,$$
and for $r=0$ 
$$P_0(\varphi)(x):=\int_{SO(n)}\tau_\la(l)\varphi(xl)dl,\ \ \text{ where } \tau_\la\in\widehat {SO(n)}.$$
For any $f\in L^1(G_n)$ and $\xi\in L^2(SO(n))$ we have that 
\begin{eqnarray*}
& & \tau_{\mu,r}(f)\xi(x):=\int_{SO(n)}\left(\int_{K_{\chi_r}}(\widehat f^2(xly^{-1},y\cdot\chi_r)\rho_\mu(l)dl\right)\xi(y)dy,\ x\in SO(n).\\
& &\tau_{\mu,0}(f)\xi(x):= \int_{SO(n)}\widehat f^2(y,0)\xi(y^{-1}x)dy,\ x\in SO(n).
\end{eqnarray*}

 Let $\D_n$ be the family consisting of all operator fields $F\in \ell^\iy(\widehat{G_n})$ satisfying the following conditions:
 \begin{enumerate}
  \item $F(\ga)$ is a compact operator on $\H_{(\ga)}$ for every $\ga\in\Gamma_1,$
  \item $\underset{\ga\to\iy}{\lim}\noop{F(\ga)}=0,$
  \item \begin{eqnarray*}
   \underset{r\to0}{\lim}\noop{F(\mu,r)\circ P_r-F(\mu,0)\circ P_0}=0.
 \end{eqnarray*}
 \end{enumerate}

Then, the $C^*$-algebra of the group {the group $G_{n} $} is isomorphic to $\D_{n}$ under the Fourier transform.

\end{example}

\begin{example}\label{tinfty}
Define the abelian group $A $ and the compact groups $L $ and $K $ by: 
\rm  
\begin{eqnarray*}\label{}
  \nn A &:= &
  \{ (z_i)_{i\in\N}\in \Z^\iy\vert z_i\ne 0\text{ for a finite number of 
indices}\}\\
\nn   & &\text{ with the discrete topology}.\\
\nn  L&:= &
\{ 1,-1\}\\  
  \nn  K&:= &
L^{\iy}\\
\nn  & &\text{ with the product topology}. 
  \end{eqnarray*}
Then, by \cite{Rei}, Ch. 4, 2, 3.1., we have that 
\begin{eqnarray*}\label{}
  \nn \widehat{A} &= &
  \T^\iy.
  \end{eqnarray*}
  Furthermore, the spectrum $\widehat K $ of the abelian group $K $ is the set of all infinite products 
  \begin{eqnarray*}
 \ch=\ch_1\times \ch_2\times \cdots \in \{1,-1\}\times\{1,-1\}\times\cdots,
 \end{eqnarray*}
where $\ch_j=1 $ almost everywhere.

Define for $m\in\N $ the subgroups  $L^m $ and $L_m $ of $K $ by
\begin{eqnarray*}\label{}
 \nn L_m &:= &
\{1\}\times \{1\}\times\cdots\underset{\text{m-th position}}{\times\{1\}} \times\{1,-1\}\times \{1,-1\}\cdots 
 \end{eqnarray*}
and
\begin{eqnarray*}\label{}
 \nn L^m &:= &
 L\times L\times\cdots\underset{\text{m-th position}}{\times L} \times\{1\}\times \{1\}\cdots
 \end{eqnarray*}

 Then $L^m $ is isomorphic to the quotient group
 $K/L_m $ and it has $2^m $ elements. The Haar measure $dl^m $ on $L^m $ is given by
 \begin{eqnarray*}\label{}
 \nn \int_{L^m}f(l_m)dl^m &= &
\frac{1}{2^m}\left( \sum_{x\in L^m}f(x)\right),\  f\in C(L^m).
 \end{eqnarray*}

A function $f:K\to \C $ is continuous, if and only if for every $\ve>0 $, there exists $m\in \N $ such that
\begin{eqnarray*}\label{}
 \nn \sup_{x\in K, l_m\in L_m} \vert{f(x)-f(xl_m) }\vert &\leq &
 \ve.
 \end{eqnarray*}
Then the Haar measure $dk  $ on $K $ is given according to \cite{Rei}, Ch. 3, 3, example (vi) by
\begin{eqnarray*}\label{}
 \nn \int_K f(k)dk&=&
 \limm  \int_{L^m}f(l_m)dl^m.
 \end{eqnarray*}

For $\chi=(\chi_i)_{i\in\N}\in\widehat{A} $,  the stabilizer $ K_\chi$ is the 
direct product
\begin{eqnarray*}\label{}
  \nn  K_\chi&= &
  \underset{i\in\N}{\prod}K_{\chi_i}
  \end{eqnarray*}
where $K_{\chi_i}=\{ 1,-1\} $ if $\chi_i\in\{ 1,-1\} $ and $K_{\chi_i}=\{ 
1\} $ if $\chi_i\not \in \{ 1,-1\} $.

Let for $m\in\N $
\begin{eqnarray*}\label{}
  \nn \chi_m &=&
  (\chi^m_j)_{j\in\N} \\
  &\text{where}&\\
  \nonumber \chi^m_j&=&
  \Big\{\begin{array}{cc}
\frac{1}{m} &\text{ if }j\leq m,\\
  0&\text{ if }j>m.
\end{array}
  \end{eqnarray*}
Then
\begin{eqnarray*}\label{}
  \nn \limm {\chi_m} &= &
  1
  \end{eqnarray*}
in $\widehat{A} $.
The stabilizer $K_{\ch_m} $ in $K $ is the subgroup
\begin{eqnarray*}
 K_{\ch_m}=\{1\}\times \{1\}\times\cdots\underset{\text{m-th position}}{\times\{1\}} \times\{1,-1\}\times \{1,-1\}\cdots 
 \end{eqnarray*}
 Then 
\begin{eqnarray*}\label{}
  \nn \limm K_{\chi_m}=1=\{ 1\}\ti\{ 1\}\ti\{ 1\}\times\cdots
  \end{eqnarray*}
Choose for every $m\in\N $ the trivial character $\mu_m $ of $K_{\ch_m} $.
Then, according to Definition \ref{Pnuk def}, the dimensions $d_{\mu_m} $ are one and  the projections 
$P_{\mu_m}:\l2K\to \l2K $ are given by
\begin{eqnarray*}\label{}
 \nn P_{\mu_m}(\va)(x) &:= &
 \int_{K_{\ch_m}}(\va(xl_m))dl^m, \va\in C(K), x\in K.
 \end{eqnarray*}
Therefore, by Remark  \ref{a converging net}, the limit set of the sequence $(\pi_{\mu_m,\ch_m}) $ is the spectrum $\widehat{K}
 $ itself and then for any $f\in\l1G $ and $\xi\in\l2K $ we have that
 \begin{eqnarray*}\label{}
 \nn \ta_{\mu_m,\chi_m}(f)\xi(x) &:= &
 \int_{K} \Big(\int_{K_{\ch_m}}\widehat f^2(x  y\inv l, y\cdot \chi_m)dl\Big)(\xi(y)) dy, x\in K,\\
 \nn \ta_{\mu_\iy,\chi_\iy}(f)\xi(x) &:= &
 \int_{K} \widehat f^2(y, 1)\xi(y\inv x)dy, x\in K.
 \end{eqnarray*}

\end{example}

\end{document}